\documentclass[12pt, final]{amsproc}

\usepackage{amssymb,amsmath,latexsym}
\usepackage[varg]{pxfonts}
\usepackage{latexsym}
\usepackage{amssymb}
\usepackage{amscd}

\usepackage[all]{xy}
\usepackage[cp850]{inputenc}
\usepackage[mathscr]{eucal}
\newtheorem{theorem}{Theorem}[section]

\newtheorem{definition}[theorem]{Definition}

\newtheorem{proposition}[theorem]{Proposition}

\def\dom{\operatorname{dom}}
\def\codom{\operatorname{cod}}
\def\PO{\operatorname{PO}}

\newcommand{\To}{\longrightarrow}
\newcommand{\U}{\mathscr {U}}
\newcommand{\N}{\mathbb{N}}

\begin{document}

%

\title{Universal disposition is not a 3-space property}

\author{Jes\'us M. F. Castillo}
\email{castillo@unex.es}

\author{Manuel Gonz\'alez}
\email{manuel.gonzalez@unican.es}

\author{Marilda A. Sim\~{o}es}
\email{simoes@mat.uniroma1.it}

\address[affil1]{Instituto de Matem\'aticas,
Universidad de Extremadura, 06071 Badajoz, Spain}
\address[affil2]{Departamento de Matem\'{a}ticas, Universidad de Cantabria,
Avenida de los Castros s/n, 39071-Santander, Spain}
\address[affil3]{Dipartimento di Matematica "G. Castelnuovo",
Universit\'a di Roma "La Sapienza", P.le A. Moro 2,
00185 Roma, Italia}

\newcommand{\AuthorNames}{F. Author, S. Author}

\newcommand{\FilMSC}{Primary 46B07; Secondary 46B03.}
\newcommand{\FilKeywords}{Banach space of universal disposition, three-space property.}
\newcommand{\FilCommunicated}{Dragan S. Djordjevi\'c.}
\newcommand{\FilSupport}{The research of the first two authors was partially supported by project
MINCIN MTM2016-76958. The research of the first author was supported in part by Project IB16056 de
la Junta de Extremadura.}

\begin{abstract}
We prove that (almost) universal disposition and separable universal disposition of Banach spaces
are not $3$-space properties.
\end{abstract}

\maketitle

\makeatletter
\renewcommand\@makefnmark%
{\mbox{\textsuperscript{\normalfont\@thefnmark)}}}
\makeatother

\section{Introduction}

Given a class $\mathfrak M$ of Banach spaces, a Banach space $X$ is said to be of (almost) universal disposition
with respect to $\mathfrak M$ if given $A, B \in \mathfrak M$, into isometries $u: A\to X$ and $\imath:A\to B$ (and
($\varepsilon>0$) there is an into isometry (an into $(1+\varepsilon$-isometry) $u': B\to X$  such that $u=u'\imath$.
A Banach space of (almost) universal disposition with respect to the class $\mathfrak F$ of finite dimensional spaces
--respectively the class $\mathfrak S$ of separable spaces-- will be simply called a space of (almost) universal
disposition --respectively of separable universal disposition--.
The Gurariy space $\mathcal G$ \cite{gurari} is the only separable Banach space of almost universal disposition,
up to isometries \cite{lusky,kuso}.
Under {\sf CH}, all spaces of separable universal disposition having density character at most $\mathfrak c$ are
isometric (see \cite{accgm2,accgmLN}) and we will denote this unique specimen by $\mathscr F_1$.
However, there are different non isomorphic spaces of universal disposition \cite{accgmLN,castsimo}; and also
different spaces of separable universal disposition outside {\sf CH} \cite{accgmLN}.
\medskip

The Gurariy space has received extensive attention in the literature  \cite{Gurariinew,gevo,godekalt,lazarlind,p-w,woj,lusky}.
Gurariy conjectured the existence of Banach spaces of universal disposition, as well as of spaces of universal disposition
with respect to the class $\mathfrak S$ of separable spaces.
This conjecture was proved to be true in  \cite{accgm2}, where a general method to construct spaces of universal
disposition with respect to different classes was presented. In particular, it was shown that the space that Gurariy
conjectured is isometric to the Fra\"iss\'e limit, in the category of separable Banach spaces and into isometries,
constructed by  Kubi\'s \cite{kubis}. The papers \cite{cgk,bagekuba} extend the methods of \cite{accgm2,kubis} to
quasi-Banach and Fr\'echet spaces. Further aspects of spaces of universal disposition have been studied
in \cite{castsimo,castsimomore,kubisfo}.
The paper \cite{castmoreka} introduces and studies the properties of (almost) universal complemented disposition and
establishes the Kadets space \cite{kadec} as a Gurariy space in a different category.
The results of \cite{castmoreka} are extended in \cite{ccm} to cover the categories of $p$-Banach spaces but using a
Fra\"iss\'e like approach in the spirit of \cite{garbula}.
\medskip

Three-space properties is, on the other hand, a well established topic (see \cite{castgonz}). Recall that a property
$\mathcal P$ of Banach spaces is said to be a $3$-space property if whenever a subspace $Y$ and the corresponding
quotient $X/Y$ of a given Banach space $X$ have $\mathcal P$ then also the space $X$ must enjoy $\mathcal P$.
Our purpose in this paper is to show that (almost) universal disposition and separable universal disposition are not
a $3$-space properties. Because of the uniqueness of the Gurariy space this in particular means that there exists a
Banach space $\Omega$ not isomorphic to $\mathcal G$ containing a subspace isomorphic to $\mathcal G$ and such that
$\Omega/\mathcal G \simeq \mathcal G$.

\section{Background}

\subsection{Isomorphic properties of Banach spaces}

A Banach space $X$ has property $(V)$ of Pe\l czy\'nski if every operator from $X$ into any other Banach space is
either weakly compact or an isomorphism on some copy of $c_0$.  All $C(K)$ and Lindenstrauss spaces enjoy
Pe\l czy\'nski's property $(V)$ \cite{pelcunc}, which is not a 3-space property as it was shown in \cite{castgonzV}
--see also \cite{castgonz}-- and later in \cite{castsimoV}.

\subsection{Exact sequences}
An exact sequence $0 \to Y \to X \to Z \to 0$ of Banach spaces is a diagram formed by Banach spaces and linear
continuous operators in which the kernel of each operator coincides with the image of the preceding; the middle
space $X$ is also called a {\it twisted sum} of $Y$ and $Z$. By the open mapping theorem this means that
$Y$ is isomorphic to a subspace of $X$ and $Z$ is isomorphic to the corresponding quotient.
An exact sequence is said to split if the image of $Y$ in $X$ is complemented; i.e., it is equivalent to the trivial
sequence $0 \to Y \to Y \oplus Z \to Z \to 0$.
See \cite{accgmLN} for details.

\subsection{A device to construct spaces of (separable) universal disposition}

Given an isometry $u:A\to B$ and an operator $t: A\to E$ there is an extension of $t$ through $u$ at the
cost of embedding $E$ in a larger space, called the \emph{push-out} space, as it is showed in the diagram
$$
\begin{CD}
A @> u >> B\\
@Vt VV @VV t' V\\
E@> u' >> \PO
\end{CD}
$$
where $t'u=u't$. It is important to realize that the operator $u'$ is again an isometry and that $t'$ is
a contraction or an isometry if $t$ is.
Once a starting Banach space $X$ has been fixed, the input data we need for our construction are:
\begin{itemize}
\item A class $\mathfrak{M}$ of separable Banach spaces.
\item The family $\mathfrak{J}$ of all isometries acting between the elements of $\mathfrak{M}$.
\item A family $\mathfrak{L}$ of norm one $X$-valued operators defined on elements of $\mathfrak{M}$.
\end{itemize}
For any operator $s : A\to B$, we establish $\dom(s) = A$ and $\codom (s) = B$.
Notice that the codomain of an operator is usually larger than its range, and that the unique codomain of
the elements of $\mathfrak{L}$ is $X$.
Set $\Gamma=\{(u,t)\in \mathfrak{J}\times \mathfrak L: \dom u=\dom t\}$ and consider the Banach spaces of
summable families $\ell_1(\Gamma,\dom u)$ and $\ell_1(\Gamma, \codom u)$.
We have an obvious isometry
$$
\oplus\mathfrak{J} :\ell_1(\Gamma, \dom u)\To \ell_1(\Gamma, \codom u)
$$
defined by  $(x_{(u,t)})_{(u,t)\in\Gamma}\longmapsto (u(x_{(u,t)}))_{(u,t)\in\Gamma} $; and a contractive
operator
$$
\Sigma\mathfrak L :\ell_1(\Gamma, \dom u)\To X,
$$
given by $(x_{(u,t)})_{(u,t)\in\Gamma}\longmapsto \sum_{(u,t)\in\Gamma}t(x_{(u,t)})$.

Observe that the notation is slightly imprecise since both $\oplus\mathfrak{J}$ and $\Sigma\mathfrak L$
depend on $\Gamma$. We can form their push-out diagram
$$
\xymatrix{
\ell_1(\Gamma, \dom u)\ar[r]^{\oplus\mathfrak{J}} \ar[d]_{\Sigma\mathfrak L} & \ell_1(\Gamma, \codom u)\ar[d]\\
E \ar[r]^\imath &\PO
}
$$
We obtain in this way an isometric enlargement of $X$ such that for every $t: A\to X$ in $\mathfrak L$, the
operator $\imath t$ can be extended to an operator $t':B \to \PO$ through any embedding $u:A\to B$ in $\mathfrak{J}$
provided $\dom u=\dom t= A$.

In the next step we keep the family $\mathfrak{J}$ of isometries, replace the starting space $X$ by $\PO$ and
$\mathfrak L$ by a family of norm one operators $\dom u \to \PO$, $u\in \mathfrak J$, and proceed again.
\medskip

We start with $\mathcal M^0(X) = X$. The inductive step is as follows. Suppose we have constructed the directed
system $(\mathcal M^\alpha(X))_{\alpha<\beta}$, including the corresponding linking maps
$\imath_{(\alpha,\gamma)}: \mathcal M^\alpha(X)\To \mathcal M^\gamma(X)$ for $\alpha<\gamma<\beta$.
To define $\mathcal M^\beta(X)$ and the maps $\imath_{(\alpha,\beta)}: \mathcal M^\alpha(X) \To \mathcal M^\beta(X)$
we consider separately two cases, as usual: if $\beta$ is a limit ordinal, then we take $\mathcal M^\beta(X)$ as the
direct limit of the system $(\mathcal M^\alpha(X))_{\alpha<\beta}$ and
$\imath_{(\alpha,\beta)}: \mathcal M^\alpha(X) \To \mathcal M^\beta(X)$ the natural inclusion map.
Otherwise $\beta=\alpha + 1$ is a successor ordinal and we construct $\mathcal M^{\beta}(X)$ applying the push-out
construction as above with the following data:  $\mathcal M^\alpha(X)$ is the starting space, $\mathfrak J$ keeps
being the set of all isometries acting between the elements of $\mathfrak M$ and $\mathfrak L_\alpha$ is the family
of all isometries $t:S\to \mathcal M^\alpha(X)$, where $S\in \mathfrak M$.
\medskip

We then set $\Gamma_\alpha=\{(u,t)\in \mathfrak J\times \mathfrak L_\alpha:\dom u=\dom t\}$ and make the push-out
\begin{equation}\label{see0}
\begin{CD}
\ell_1(\Gamma_\alpha, \dom u)@> \oplus\mathfrak I_\alpha >> \ell_1(\Gamma_\alpha, \codom u)\\
@V \Sigma\mathfrak L_\alpha VV @VVV\\
\mathcal M^\alpha (X)@>>> \PO
\end{CD}
\end{equation}
thus obtaining ${\mathcal M}^{\alpha + 1}(X)=\PO$. The embedding $\imath_{(\alpha,\beta)}$ is the lower arrow
in the above diagram; by composition with $\imath_{(\alpha,\beta)}$ we get the embeddings
$\imath_{(\gamma,\beta)} =\imath_{(\alpha, \beta)}\imath_{(\gamma, \alpha)}$, for all $\gamma < \alpha$.
\medskip

Our construction will conclude at the first uncountable ordinal $\omega_1$ providing a Banach space that we denote
$\mathcal{M}^{\omega_1}(X)$.
\medskip

The choice $\mathfrak M = \mathfrak F$ of finite dimensional spaces will provide the space of universal disposition
$\mathscr  F^{\omega_1}(X)$. The choice $\mathfrak M = \mathfrak S$ of separable spaces will provide the space of
separable universal disposition $\mathcal S^{\omega_1}(X)$ (cf. \cite[Chapter 3]{accgmLN}).

Although spaces of universal disposition need not be isomorphic, it was proved in \cite[Theorem 3.23]{accgmLN} that
all spaces $\mathscr F^{\omega_1}(X)$ are isometric for separable $X$ and that all spaces $\mathcal S^{\omega_1}(X)$
are isometric for $X$ having density character at most $\aleph_1$.
Consequently, it makes sense to use the following notation.

\begin{definition} $\;$
\begin{itemize}
\item $\mathscr F(X)$ for $\mathscr  F^{\omega_1}(X)$. We will simply write $\mathscr F_0$ to denote $\mathscr F(X)$
for separable $X$
\item $\mathscr G(X)$ for $\mathcal  S^{\omega_1}(X)$. We will simply write $\mathscr F_1$ to denote $\mathscr G(X)$
for $X$ having density character at most $\aleph_1$.
\end{itemize}
\end{definition}

\section{Universal disposition are not a $3$-space properties}

Observe that universal disposition, as well as all the other related properties, are geometrical properties; thus,
one has that the space $\mathcal G$ is not necessarily of almost universal disposition under an equivalent renorming.
Therefore, to consider $3$-space problems one need to think about the associated isomorphic properties: namely,
to be isomorphic to a space of (almost) universal (separable) disposition.
So we will do without this further explicit mention.

\begin{theorem}
Universal and almost universal disposition are not $3$-space properties.
\end{theorem}
\begin{proof}
We get from \cite{accgmLN} that a space of almost universal disposition is a Lindenstrauss space, and therefore
it enjoys property $(V)$. It follows from \cite{ccky}, see also \cite{castsimoV}, that there exist exact sequences
$$\begin{CD}
0 @>>> C(\omega^\omega) @>>> \Omega @>>> c_0 @>>>0
\end{CD}$$
in which $\Omega$ has not property (V). Since every separable Lindenstrauss space is $1$-complemented in
$\mathcal G$ and property $(V)$ is stable by products, multiplying adequately on the left and on the right one
can thus obtain an exact sequence
$$\begin{CD}
0 @>>> \mathcal G @>>> \Omega' @>>> \mathcal G @>>>0
\end{CD}$$
in which $\Omega'$ cannot have property $(V)$ and thus it cannot be isomorphic to $\mathcal G$.
In particular, it cannot be a space of almost universal disposition.

To prove that universal disposition is not a $3$-space property, we use \cite[Proposition 2.1]{castsimo} and
\cite[Lemma 3.2]{castsimo} to get that both $C(\omega^\omega)$ and $c_0$ are complemented in the space of universal
disposition $\mathscr F(C(\omega^\omega))$.
Therefore, multiplying adequately on the left and on the right one can obtain an exact sequence
$$\begin{CD}
0 @>>> \mathscr F(C(\omega^\omega)) @>>> \Omega'' @>>> \mathscr F(C(\omega^\omega)) @>>>0
\end{CD}$$
namely
$$\begin{CD}
0 @>>> \mathscr F_0 @>>> \Omega'' @>>> \mathscr F_0 @>>>0
\end{CD}$$
in which $\Omega''$ contains $\Omega$ complemented and thus it cannot have property $(V)$, which prevents it to be
a space of universal disposition.
\end{proof}

The strategy to show that separable universal disposition is not a 3-space property has to be different since
spaces of separable universal disposition do not contain complemented separable subspaces.

\begin{theorem}
Under {\sf CH}, separable universal disposition is not a $3$-space property.
\end{theorem}
\begin{proof}
We need here the results of \cite{accgm2} asserting that a space of separable universal disposition enjoy a property
called universal separable injectivity \cite{accgmLN}.
A Banach space $X$ is universally separably injective if every operator $S\to X$ from a separable space $S$ can be
extended to any superspace.
It is obvious that complemented subspaces of universally separably injective spaces enjoy the same property.
Ultrapowers of $\mathcal L_\infty$ spaces are also universally separably injective.
The main result of \cite[Theorem 1 (f)]{accgmLNc} asserts that, under {\sf CH}, universal separable injectivity is
not a $3$-space property. More precisely, there exists an ultrapower $C[0,1]_\U$ of $C[0,1]$ and an exact sequence
$$\begin{CD}
0 @>>> C[0,1]_\U @>>> \Psi @>>> \ell_\infty@>>>0
\end{CD}$$
in which the space $\Psi$ is not universally separably injective.
Now, since $C[0,1]$ is complemented in $\mathcal G$ then $C[0,1]_\U$ is complemented in $\mathcal G_\U$.

On the other hand, $\mathcal G_\U$, as any ultrapower of an $\mathcal L_\infty$ space, contains $\ell_\infty$,
and it must then contain it complemented.
Thus, multiplying adequately on the left and on the right one gets an exact sequence
$$\begin{CD}
0 @>>> \mathcal G_\U @>>> \Psi\oplus \mathcal G_\U @>>> \mathcal G_\U@>>>0
\end{CD}$$
in which the twisted sum space $\Psi\oplus \mathcal G_\U $ cannot be of separable universal disposition since
it is not even universally separably injective.
\end{proof}

\section{Stability by products}

Since the properties of (almost) universal disposition and separable universal disposition are not $3$-space
properties, it makes sense to consider their (isomorphic) stability by products.
We have been unable to determine whether the product of two spaces of (almost) universal disposition (or
separable universal disposition) has to be isomorphic to a space with the same property.
Part of the problem is that there is no available characterization for the property ``to be isomorphic to a
space of (almost) universal disposition".  One however has:

\begin{proposition}$\;$
\begin{enumerate}
\item $\mathcal G \simeq \mathcal G \oplus \mathcal G$.
\item Under the diamond axiom $\diamondsuit$ (see \cite{shelah}) there is a space $\mathcal S_\diamondsuit$
of density character $\aleph_1$ and of almost universal disposition such that $\mathcal S_\diamondsuit$ is
not isomorphic to $\mathcal S_\diamondsuit \oplus \mathcal S_\diamondsuit$.
However, any ultrapower $(\mathcal S_\diamondsuit)_\U$ is a space of separable universal disposition such that
$(\mathcal S_\diamondsuit)_\U \simeq (\mathcal S_\diamondsuit)_\U \oplus (\mathcal S_\diamondsuit)_\U$.
\item Under {\sf CH},   $\mathscr F_1 \oplus \mathscr F_1$ is isomorphic to a space of separable universal disposition.
\end{enumerate}
\end{proposition}
\begin{proof}
(1) That  $\mathcal G \oplus \mathcal G\simeq \mathcal G$ can be found in \cite[p.133]{accgmLN}: since
$c_0(\mathcal G)$ is a Lindenstrauss space, it is $1$-complemented in $\mathcal G$; and $\mathcal G$ is
obviously $1$-complemented in $c_0(\mathcal G)$.
By Pe\l cy\'nski's decomposition method $c_0(\mathcal G) \simeq \mathcal G$ and, in particular,
$\mathcal G \oplus \mathcal G\simeq \mathcal G$.
\medskip

(2) The space $\mathcal S_\diamondsuit$ is the one constructed by Shelah \cite[Theorem 5.1]{shelah}.
It has the property labeled (C) that every operator $\mathcal S_\diamondsuit \to \mathcal S_\diamondsuit $ has
the form $\lambda I + S$ for some scalar $\lambda$ and some separable range operator $S$.
This makes impossible an isomorphism $\mathcal S_\diamondsuit \oplus \mathcal S_\diamondsuit \simeq \mathcal S_\diamondsuit$.

The space $\mathcal S_\diamondsuit$ is of almost universal disposition by construction, and consequently
$(\mathcal S_\diamondsuit)_\U$ is of separable universal disposition. Since the diamond axiom $\diamondsuit$ implies {\sf CH},
$\mathscr F_1 \simeq (\mathcal S_\diamondsuit)_\U$, and we prove next that $\mathscr F_1 \oplus \mathscr F_1 \simeq \mathscr F_1$.
\medskip

(3) Under {\sf CH}, all spaces of separable universal disposition having density character at most $\aleph_1$ are isometric
\cite{accgmLN}. Therefore $\mathscr F_1 \simeq \mathcal G_\U$ for every countably incomplete ultrafilter on $\N$. Thus,
$$
\mathscr F_1 \simeq \mathcal G_\U \simeq (\mathcal G \oplus \mathcal G)_\U \simeq \mathcal G _\U \oplus \mathcal G _\U \simeq
 \mathscr F_1 \oplus \mathscr F_1
$$
\end{proof}

As we have already said, we have been unable to settle the following
\medskip

\noindent \textbf{Problem.}
\emph{Is the product of spaces of universal disposition (isomorphic to) a space of universal disposition?}
\medskip

Observe that, even under {\sf CH}, there are at least three different non isomorphic spaces of universal disposition:
$\mathscr F_1, \mathscr F_0$ (see \cite{accgmLN}) and $\mathscr F(C(\Delta))$, where $C(\Delta)$ is a space of continuous
functions on a compact that admits a representation
$$\begin{CD}
0 @>>> c_0 @>>> C(\Delta) @>>> c_0(\aleph_1)@>>>0
\end{CD}$$
(see \cite{castsimo}). The spaces $\mathscr F_1$ and $\mathscr F_0$ are very different since $\mathscr  F_1$ is $1$-separably
injective while $\mathscr F_0$ is not even separably injective; $\mathscr F_1$ is a Grothendieck space while every copy of
$c_0$ inside $\mathscr F_0$ is complemented.

The space $\mathscr F(C(\Delta))$ is different from the other two since it contains complemented and uncomplemented copies
of $c_0$. It is clear that $\mathscr F_1 \oplus \mathscr F_0$ cannot be isomorphic to either $\mathscr F_1$ (since it cannot
be separably injective) or to $\mathscr F_0$ (since it contains uncomplemented copies of $c_0$).
Is it a space of universal disposition?

We conjecture that it is not and that, consequently, the answer to the Problem is negative. Recall that it is also an open
problem whether there exist a continuum of mutually non-isomorphic spaces of universal disposition.
It is also likely that the spaces $\mathscr F(\mathscr F_0)$, $\mathscr F(\mathscr F_1)$ and $\mathscr F(\mathscr F(C(\Delta)))$
are non isomorphic, which would yield an infinite sequence of non-isomorphic spaces of universal disposition.


\end{document}